\documentclass[]{interact}
\usepackage{graphicx}%
\usepackage{subfigure}
\usepackage{multirow}%
\usepackage{amsmath,amssymb,amsfonts}%
\usepackage{amsthm}%
\usepackage{mathrsfs}%
\usepackage[title]{appendix}%
\usepackage{xcolor}%
\usepackage{textcomp}%
\usepackage{manyfoot}%
\usepackage{booktabs}%
\usepackage{algorithm}%
\usepackage{algorithmicx}%
\usepackage{algpseudocode}%
\usepackage{listings}%
\usepackage{CJK}
\usepackage{indentfirst}
\usepackage{cases}
\usepackage{longtable}
\usepackage{array}
\usepackage{relsize}
\usepackage{url}

\usepackage{soul} 
\usepackage{color, xcolor} 

\usepackage{epstopdf}
\usepackage[caption=false]{subfig}

\usepackage[numbers,sort&compress]{natbib}
\bibpunct[, ]{[}{]}{,}{n}{,}{,}
\makeatletter
\def\NAT@def@citea{\def\@citea{\NAT@separator}}
\makeatother

\theoremstyle{plain}
\newtheorem{Theorem}{Theorem}

\newtheorem{Lemma}{Lemma}
\newtheorem{Proposition}{Proposition}

\theoremstyle{definition}
\newtheorem{Definition}{Definition}

\newtheorem{Assumption}{Assumption}

\theoremstyle{remark}
\newtheorem{Remark}{Remark}

\begin{document}
	
	
\title{Exact-Penalty Prox-Linear Methods for Bilevel Optimization with $\ell_1$ Lower-Level Gradient Penalty}
	
\author{
\name{Yutong Zheng\textsuperscript{a}, Jiani Li\textsuperscript{a}
and
Qingna Li\textsuperscript{a,b}\thanks{Email: qnl@bit.edu.cn. Corresponding author. } 
}
\affil{\textsuperscript{a}School of Mathematics and Statistics, Beijing Institute of Technology, Beijing, China \textsuperscript{b}Beijing Key Laboratory on MCAACI/Key Laboratory of Mathematical Theory and Computation in Information Security, Beijing Institute of Technology, Beijing, China
}
	}
	
\maketitle
	
\begin{abstract}
Bilevel optimization is a fundamental framework for hierarchical decision-making, but its solution is challenging due to the implicit and typically set-valued nature of the lower-level optimality condition. In this paper, 
	we study bilevel optimization problems through an exact-penalty reformulation based on the $\ell_1$-norm of the lower-level gradient. Under suitable regularity assumptions, 
	we show that this penalty defines a distance-bound function and yields an exact penalty property for sufficiently large penalty parameters.
	To solve the resulting nonsmooth penalized problem, we propose an exact-penalty prox-linear (EPPL)  method and establish a stationarity-oriented convergence guarantee. We further specialize the method to the simple bilevel setting, where the subproblem admits an explicit dual reformulation as a box-constrained quadratic program. This structure leads to a dual spectral projected gradient method with closed-form primal recovery, for which convergence of the inner dual iterates is proved. 
	Numerical experiments on a minimum-norm least-squares bilevel model show that the proposed method is effective in reducing both the lower-level and upper-level gaps to high accuracy. 
	Compared with several existing methods, the proposed approach attains the best final solution accuracy on the tested instance.
\end{abstract}
	

\section{Introduction}\label{sec1}
Bilevel optimization concerns hierarchical decision problems in which the feasible set of an upper-level problem is determined by the solution set of a lower-level optimization problem. Owing to this nested structure, bilevel models arise in a broad range of areas, including hyperparameter optimization, meta-learning, signal and image processing \cite{colson2007overview,sinha2018review,franceschi2018bilevel}. 

Bilevel optimization is a classical topic in mathematical programming. 
In machine learning, bilevel formulations now play a central role in hyperparameter optimization, meta-learning, neural architecture search, and other hierarchical learning pipelines. Representative examples include the bilevel framework for hyperparameter selection and meta-learning \cite{franceschi2018bilevel}, truncated back-propagation through lower-level dynamics \cite{shaban2019truncated}, implicit-gradient-based meta-learning \cite{rajeswaran2019meta}, the generic first-order framework beyond the lower-level singleton assumption \cite{liu2020generic}, and fully first-order stochastic bilevel methods such as \cite{kwon2023fully}. For nonsmooth lower-level learning models, Ochs et al.\ developed early gradient-based techniques that have had considerable influence on later bilevel optimization methods \cite{ochs2016nonsmooth}. 
Alongside these general-purpose machine-learning-oriented bilevel methods, another line of work studies problem-specific reformulations, especially for support vector classification and related learning models. For hyperparameter selection in $\ell_1$-loss support vector classification, Li, Li, and Zemkoho \cite{li2022svc} formulated the cross-validation problem as a bilevel program, converted it into an MPEC, and proposed a global-relaxation-based solution method. This MPEC perspective was further developed in \cite{li2025mpeccq}, where several classical constraint qualifications were analyzed in the context of bilevel hyperparameter optimization for machine learning. For multiple hyperparameter selection with feature selection in support vector classification, Qian, Li, and Zemkoho \cite{qian2023grlpn} proposed a global relaxation-based LP--Newton method. For support vector classification with logistic loss, Wang and Li \cite{wang2025logisticnewton} transformed the bilevel model into a KKT-based nonlinear program and developed a fast smoothing Newton method with superlinear local convergence. For large-scale MPECs arising from hyperparameter selection in $\ell_1$-support vector classification, Wang, Li, and Zhang \cite{wang2025sdnm} proposed a smoothing damped Newton method that directly exploits problem structure. Beyond hyperparameter selection, Zheng and Li \cite{zheng2025adversarial} studied bilevel models for adversarial learning and analyzed a case study in convex clustering. These works further illustrate the breadth of bilevel optimization in machine learning, while also showing that many successful approaches are strongly tailored to specific application structures and KKT/MPEC reformulations.

Penalty methods provide another important route for handling bilevel structure without explicitly differentiating through the lower-level solution map. A classical reference is the work of Ye, Zhu, and Zhu \cite{ye1997exact}, which established exact penalization and necessary optimality conditions for generalized bilevel programming. In the modern first-order literature, Lu and Mei \cite{lu2024penalty} studied penalty methods for a broad class of unconstrained and constrained bilevel problems by solving structured minimax subproblems. Shen, Xiao, and Chen \cite{shen2025pbgd} developed penalty-based bilevel gradient descent for constrained bilevel problems without lower-level strong convexity. In the simple bilevel setting, Chen et al.\ \cite{chen2024holder} further connected approximate solutions of the original problem and its penalized reformulation under H\"olderian error bounds. 

For simple bilevel optimization, a substantial body of work has focused on first-order algorithms tailored to convex and structured settings. Classical approaches include explicit descent methods \cite{solodov2007explicit}, minimal-norm selection methods \cite{beck2014minimal}, and $\epsilon$-subgradient schemes \cite{helou2017epsilon}. Subsequent developments introduced projected and incremental regularization strategies \cite{amini2019iterative}, methods for variational-inequality-constrained formulations \cite{kaushik2021vi}, and inertial or fixed-point type algorithms \cite{sabach2017first,shehu2021inertial}. More recently, the algorithmic scope has been broadened to accommodate nonsmooth outer objectives and more general first-order frameworks; see, for example, the ITALEX methodology of Doron and Shtern \cite{doron2023italex}, the online-convex-optimization-based framework of Shen, Ho-Nguyen, and K{\i}l{\i}n\c{c}-Karzan \cite{shen2023oco}, and the Bi-SG method of Merchav and Sabach for nonsmooth outer-level objectives \cite{merchav2023nonsmooth}.

A particularly active recent line of research studies non-asymptotic guarantees for convex simple bilevel optimization under weaker structural assumptions. Jiang et al.\ \cite{jiang2023cg} proposed a conditional-gradient method for simple bilevel problems with convex lower-level structure, while Cao et al.\ \cite{cao2024accelerated} developed an accelerated method for the convex smooth case. Wang, Shi, and Jiang \cite{wang2024bisection} established near-optimal complexity guarantees via a bisection framework, and Zhang et al.\ \cite{zhang2024fcbio} studied a functionally constrained reformulation and derived near-optimal guarantees under standard convex smoothness assumptions. In addition, Chen et al.\ \cite{chen2024holder} analyzed penalty-based methods under H\"olderian error bounds, and Giang-Tran, Ho-Nguyen, and Lee \cite{giangtran2025projectionfree} proposed a projection-free conditional-gradient method with simultaneous inner and outer convergence guarantees.

In this paper, we consider a bilevel optimization problem in which the upper-level decision variable is constrained by the solution set of a lower-level optimization problem. Rather than enforcing the lower-level optimality condition explicitly, we measure its violation through the $\ell_1$-norm of the lower-level gradient and incorporate this quantity into the upper-level objective as a penalty term. This leads to a penalized formulation that combines the original upper-level objective with a weighted lower-level gradient penalty.

To solve the resulting nonsmooth penalized problem, we propose a prox-linear method motivated by \cite{shen2025pbgd}. 
The common feature is that both works start from a penalty reformulation of bilevel optimization and use the lower-level gradient to construct a single-level penalized objective. Moreover, both works recognize that the nonsmooth penalized problem can be treated through a prox-linear type model.
The main difference lies in the modeling focus and the resulting algorithmic structure. The work \cite{shen2025pbgd} studies penalty-based bilevel gradient descent for constrained bilevel problems and develops a general penalty-based prox-linear treatment for its nonsmooth penalized formulation. In contrast, our paper focuses on an exact-penalty reformulation based on the $\ell_1$-norm of the lower-level gradient and specializes the method to the simple bilevel setting. This specialization allows us to derive an explicit dual reformulation of the prox-linear subproblem as a box-constrained quadratic program, which further leads to a dedicated dual SPG solver with nonmonotone line search and closed-form primal recovery. Therefore, our method should be viewed as a problem-specific exact-penalty prox-linear method that exploits the special dual geometry induced by the $\ell_1$ penalty.

The main contributions of this paper are summarized as follows.
\begin{itemize}
	\item We study an exact-penalty reformulation of bilevel optimization based on the $\ell_1$-norm of the lower-level gradient. Under suitable regularity assumptions, 
	we show that this penalty defines a distance-bound function and yields an exact reformulation for sufficiently large penalty parameters.
	
	\item For the resulting nonsmooth penalized problem, we develop an exact-penalty prox-linear method (EPPL) motivated by \cite{shen2025pbgd}. By linearizing the smooth part and the lower-level gradient mapping while keeping the $\ell_1$-structure of the penalty term, we obtain a structured strongly convex subproblem and establish a stationarity-oriented convergence guarantee for the outer iterations.
	
	\item We further specialize the method to the simple bilevel problem. In this setting, we derive an explicit dual reformulation of the subproblem as a box-constrained quadratic program, characterize its first-order optimality condition through a projected fixed-point equation. 
	
	\item Based on the dual box geometry, we design a practical spectral projected gradient (SPG) solver for the subproblem of simple bilevel problem, together with closed-form primal recovery. Numerical experiments on the minimum-norm least-squares bilevel model show that the proposed method is competitive with existing approaches and attains the best final solution accuracy among the tested methods.
\end{itemize}

The rest of the paper is organized as follows. Section~\ref{sec:penalty} introduces the exact-penalty reformulation of the bilevel problem based on the $\ell_1$-norm of the lower-level gradient and establishes its exactness under suitable assumptions. Section~\ref{sec:proxlinear} presents the exact-penalty prox-linear method for the penalized problem and provides its convergence analysis. Section~\ref{sec:simple} focuses on the simple bilevel problem, derives the dual reformulation of the subproblem.  Section~\ref{sec:dual-subproblem} develops the dual SPG solver and together with its convergence guarantee. Section~\ref{sec:numerics} reports numerical results on the YearPredictionMSD linear regression instance and compares the proposed method with several existing algorithms. Finally, Section~\ref{sec:conclusion} concludes the paper.

\textbf{Notations.} We use $\|\cdot\|$ to denote the $\ell_2$
norm. Given a nonempty closed set $\mathcal{S}\subseteq\mathbb{R}^d$, define the distance of $y\in\mathbb{R}^d$ to the set $\mathcal{S}$ by
$
d_{\mathcal{S}}(y):=\min_{y'\in\mathcal{S}}\|y-y'\|. 
$

\section{Exact Penalty Reformulation of Bilevel Problems}\label{sec:penalty}
This section studies the relation between the solutions of the bilevel problem (BP) and those of norm-gradient   penalized problem (NGP) by posting certain generic conditions.

Define
$
f:\mathbb{R}^{d_x}\times\mathbb{R}^{d_y}\to\mathbb{R}
\ \text{and}\ 
g:\mathbb{R}^{d_x}\times\mathbb{R}^{d_y}\to\mathbb{R}.
$
We consider the following bilevel problem: 
\begin{equation}\label{eq:BP}
	\min_{x,y}\ f(x,y)
	\quad \text{s.t.}\quad
	x\in C,\quad y\in S(x),
	\tag{BP}
\end{equation}
where
\[
S(x):=\arg\min_{u\in\mathbb R^{d_y}} g(x,u).
\]
Here $C\subseteq \mathbb R^{d_x}$ is a nonempty closed convex set, and $f,\ g$ are continuously differentiable. For every
$x\in C$, the lower-level solution set $S(x)$ is assumed to be nonempty and
closed.

\begin{Assumption}\label{ass:PL}
	For every $x\in C$, the function $g(x,\cdot)$ is continuously differentiable
	on $\mathbb R^{d_y}$ and satisfies the Polyak--\L{}ojasiewicz (PL) inequality with
	constant $1/\sqrt{\mu}$, namely,
	\[
	\|\nabla_y g(x,y)\|_2^2
	\ge
	\frac{1}{\sqrt{\mu}}\bigl(g(x,y)-v(x)\bigr),
	\qquad
	\forall y\in\mathbb R^{d_y},
	\]
	where 
	$
	v(x):=\min_{u\in\mathbb R^{d_y}}g(x,u).
	$
\end{Assumption}

Under Assumption~\ref{ass:PL}, the implicit lower-level optimality condition
$y\in S(x)$ can be equivalently replaced by the lower-level gradient equation.
Indeed, by Fermat's rule \cite[Theorem~10.1]{rockafellar1998variational},
every global minimizer of the unconstrained differentiable lower-level problem
satisfies the stationarity condition $\nabla_y g(x,y)=0$.  Conversely, the PL
condition implies that every stationary point is a global minimizer
\cite{karimi2016linear}.
Hence, for every $x\in C$,
\[
y\in S(x)
\quad \Longleftrightarrow \quad
\nabla_y g(x,y)=0.
\]
Therefore, \eqref{eq:BP} is equivalent to the following gradient-based
reformulation:
\begin{equation}\label{eq:BP2}
	\min_{x,y}\ f(x,y)
	\quad \text{s.t.}\quad
	x\in C,\quad \nabla_y g(x,y)=0.
	\tag{GBP}
\end{equation}

A central difficulty in \eqref{eq:BP} is that the lower-level optimality
condition $y\in S(x)$ is implicit and typically set-valued. Under
Assumption~\ref{ass:PL}, this condition can be equivalently represented by the
gradient equation in \eqref{eq:BP2}. Nevertheless, \eqref{eq:BP2} is still a
nonlinear equality-constrained problem, and directly solving it would require
handling the constraint $\nabla_y g(x,y)=0$ and its associated Jacobian, which
involves second-order information of the lower-level objective. 
To avoid enforcing this equality constraint explicitly, we adopt an
exact-penalty reformulation. Motivated by the classical theory of nonsmooth
exact penalties and recent penalty-based first-order methods for bilevel
optimization \cite{ye1997exact,lu2024penalty,shen2025pbgd}, we penalize the gradient equation as follows
\[
p(x,y):=\|\nabla_y g(x,y)\|_1.
\]
For a penalty parameter $\gamma>0$, we consider the norm-gradient penalized
problem
\begin{equation}\label{eq:BP-pen}
	\min_{x,y}\ F_\gamma(x,y)
	:=
	f(x,y)+\gamma p(x,y)
	\quad \text{s.t.}\quad
	x\in C,\quad y\in\mathbb R^{d_y}. 
	\tag{NGP}
\end{equation}

To analyze the relation of \eqref{eq:BP2} and \eqref{eq:BP-pen}, we need the following assumptions and definitions. 

\begin{Assumption}
	\label{ass:Lip-f}
	There exists a constant $L>0$ such that for any $x\in C$, the mapping
	$y \mapsto f(x,y)$ is $L$-Lipschitz continuous on $\mathbb R^{d_y}$, i.e.,
	\[
	|f(x,y)-f(x,y')|
	\le L \|y-y'\|,
	\qquad \forall\, y,y'\in \mathbb R^{d_y}.
	\]
\end{Assumption}

\begin{Definition}
	A function $p:\mathbb{R}^{d_x}\times\mathbb{R}^{d_y}\to\mathbb{R}$ is a $\rho$-distance-bound if there exists $\rho>0$ such that for any $x\in C$ and $y\in \mathbb R^{d_y}$, it holds
	\begin{align}
		p(x,y) &\ge 0, \qquad 	d_{S(x)}(y)\le \rho\, p(x,y), \tag{a}\\
		p(x,y) &=0\iff d_{S(x)}(y)=0. \tag{b}
	\end{align}
\end{Definition}

Based on the above assumptions and definitions, we have the following results. 
\begin{Lemma}\label{lem:l1-distance-bound}
	Suppose Assumption~\ref{ass:PL} holds. 
	Then $p(x,y)=\|\nabla_y g(x,y)\|_1$ is a $\rho$-distance-bound function with $\rho := \sqrt{\mu}$. 
\end{Lemma}
\begin{proof}
The nonnegativity of $p$ is immediate. Since the lower-level problem is
unconstrained in $y$, for every fixed $x\in C$ we have
\[
y\in S(x)
\quad\Longrightarrow\quad
\nabla_y g(x,y)=0.
\]
Conversely, if $p(x,y)=0$, then $\nabla_y g(x,y)=0$. By the PL inequality,
\[
0=\|\nabla_y g(x,y)\|_2^2
\ge
\frac{1}{\sqrt{\mu}}\bigl(g(x,y)-v(x)\bigr).
\]
Since $g(x,y)\ge v(x)$, it follows that $g(x,y)=v(x)$, and hence
$y\in S(x)$. Therefore,
\[
p(x,y)=0
\quad\Longleftrightarrow\quad
d_{S(x)}(y)=0.
\]

Moreover, by \cite[Lemma 2]{shen2025pbgd}, under the PL condition with
constant $1/\sqrt{\mu}$, the function
$
\|\nabla_y g(x,y)\|_2^2 
$
is a $\mu$-squared-distance-bound function. Hence
\[
d_{S(x)}(y)^2
\le
\mu \|\nabla_y g(x,y)\|_2^2.
\]
Taking square roots gives
\[
d_{S(x)}(y)
\le
\sqrt{\mu}\|\nabla_y g(x,y)\|_2.
\]
Since $\|a\|_2\le \|a\|_1$ for any vector $a$, we obtain
\[
d_{S(x)}(y)
\le
\sqrt{\mu}\|\nabla_y g(x,y)\|_1
=
\rho p(x,y),
\]
with $\rho=\sqrt{\mu}$. Therefore, $p$ is a $\rho$-distance-bound function.
\end{proof}

\begin{Remark}
	The choice of 
	$
	p(x,y):=\|\nabla_y g(x,y)\|_1
	$ 
	is motivated by both exact-penalty theory and algorithmic tractability. 
	From the modeling viewpoint, $p(x,y)$ is a computable measure of the violation
	of the gradient equation in \eqref{eq:BP2}. Under Assumption~\ref{ass:PL}, the
	condition $p(x,y)=0$ is equivalent to $y\in S(x)$. 
	From the algorithmic viewpoint, the $\ell_1$-norm is particularly attractive. 
	Although it is nonsmooth, it admits a simple  dual representation, and after linearization this structure leads to a box-constrained quadratic subproblem in the dual space. 
	This dual geometry is central to our subsequent algorithmic development: it allows us to combine an exact-penalty reformulation with a prox-linear outer scheme and an efficient dual spectral projected gradient inner solver.
\end{Remark}

Now we can prove the exactness of the penalized problem \eqref{eq:BP-pen}. The following result clarifies the relation between minimizers of the problem \eqref{eq:BP} and minimizers of the penalized problem \eqref{eq:BP-pen}, and therefore provides the theoretical justification for our algorithmic framework.

\begin{Theorem}\label{prop:l1-global}
	Suppose Assumptions~\ref{ass:PL}-\ref{ass:Lip-f} hold,  $g(x, \cdot)$ is Lipschitz-smooth, and
	$p(x,y)=\|\nabla_y g(x,y)\|_1$ satisfies Lemma~\ref{lem:l1-distance-bound}
	with $\rho=\sqrt{\mu}$.
	For any $\gamma > L\rho = L\sqrt{\mu},$ the following arguments hold. 
\begin{itemize}
	\item [(i)] The global solutions of \eqref{eq:BP-pen} are the global solutions of \eqref{eq:BP}. 
	\item [(ii)] If $\|\nabla_y g(x,y)\|_1$ is convex in $y$, then the local solutions of \eqref{eq:BP-pen} are the local solutions of \eqref{eq:BP}. 
\end{itemize}
\end{Theorem}
%
\begin{proof}
By Lemma~\ref{lem:l1-distance-bound}, $p(x,y)$ is a $\rho$-distance-bound function.
Hence the general exact-penalty statement in \cite[Proposition~6]{shen2025pbgd} 
applies with $\gamma>\gamma^{\ast}=L\rho$.
\end{proof}

\begin{Remark}
	\label{rem:NW-exact-penalty}
	The penalized problem \eqref{eq:BP-pen} can also be viewed as the classical
	$\ell_1$ exact-penalty reformulation of the equality-constrained problem
	\eqref{eq:BP2}. 
	If \eqref{eq:BP2} is locally in the
	standard nonlinear programming form around a strict local solution
	$(\bar x,\bar y)$, and if the first-order necessary conditions hold with a
	Lagrange multiplier $\eta^\ast\in\mathbb R^{d_y}$ associated with the equality
	constraint $\nabla_y g(x,y)=0$, then
	Nocedal and Wright~\cite[Theorem~17.3]{nocedal2006numerical} imply that
	$(\bar x,\bar y)$ is a local minimizer of \eqref{eq:BP-pen} whenever
	$
	\gamma>\|\eta^\ast\|_\infty. 
	$
	In addition, if the second-order sufficient conditions hold at
	$(\bar x,\bar y)$, then $(\bar x,\bar y)$ is a strict local minimizer of
	\eqref{eq:BP-pen}.
	
\end{Remark}

The details of the Exact-Penalty Method are given in Algorithm \ref{alg:outer-continuation-0}. 

\begin{algorithm}[H]
	\caption{Exact-Penalty Method for \eqref{eq:BP-pen}}\label{alg:outer-continuation-0}
	\begin{algorithmic}[1]
		
		\State \textbf{Input:} $(x^0,y^0)$ and $\gamma_0>0$, $\tau>1$, $\varepsilon_{f}>0$, $\varepsilon_{s}>0$. 
		
		\State Set $k:=0$.
		
		\State \textbf{while} 	$r_{f}^{k}> \varepsilon_{f}\ \text{or}\ 
		r_{s}^{k}> \varepsilon_{s}$, \textbf{do}
		
		\State \quad Solve 
		\begin{equation}\label{eq:EP_gammak_0}
				\min_{x\in C,y\in\mathbb R^{d_y}}\ F_{\gamma_k}(x,y)
			:=
			f(x,y)+\gamma_k p(x,y). 
		\end{equation}
		\quad to get $(x^{k+1},y^{k+1})$.
		
		\State \quad Set 
		$\gamma_{k+1}:=\tau \gamma_{k}$ and $k:=k+1$
		
		\State \textbf{end while}
		
		\State \textbf{Output:} $(\bar x,\bar y)=(x^{k},y^{k})$.
		
	\end{algorithmic}
\end{algorithm}

We terminate Exact-Penalty Method by combining a feasibility residual for the gradient equation
in \eqref{eq:BP2} and a stationarity residual. Specifically, after
obtaining $(x^{k+1},y^{k+1})$, we compute
\begin{equation}\label{eq:EPPL-residuals}
	r_{f}^{k+1}
	:=
	\|\nabla_y g(x^{k+1},y^{k+1})\|_1,
	\qquad
	r_{s}^{k+1}
	:=
	\frac{1}{\lambda_k}\|(x^{k+1},y^{k+1})-(x^{k},y^{k})\|_2 .
\end{equation}
The first residual measures the violation of the equality constraint
$\nabla_y g(x,y)=0$ in \eqref{eq:BP2}, while the second residual measures the
stationarity of the step. The algorithm stops once
\begin{equation}\label{eq:EPPL-stop}
	r_{f}^{k+1}\le \varepsilon_{f},
	\qquad
	r_{s}^{k+1}\le \varepsilon_{s}.
\end{equation}
If \eqref{eq:EPPL-stop} is not satisfied, the penalty parameter is increased. 

\section{Solving the Penalized Problem }\label{sec:proxlinear}

In this section, we develop an algorithmic framework for solving the penalized problem \eqref{eq:BP-pen}. 
Our goal is to show how the special structure of the exact-penalty objective can be exploited algorithmically. 
In particular, we adapt the prox-linear framework of \cite{drusvyatskiy2019efficiency} to our bilevel penalty formulation and then derive a problem-specific subproblem that will be further exploited in the later sections.

Solving \eqref{eq:BP-pen} is challenging because it contains the nonsmooth penalty term $\|\nabla_y g(x,y)\|_1$, which arises from the nonsmoothness of the $\ell_1$-norm. Although one may smooth this term by using, for example, the Moreau envelope or zeroth-order approximations, such smoothing techniques may deteriorate the convergence rate. Motivated by the simple structure of the $\ell_1$-norm, we instead solve \eqref{eq:BP-pen} via a prox-linear approach.

We first recall the generic prox-linear framework introduced in \cite{drusvyatskiy2019efficiency}. 
The prox-linear method is designed for composite nonsmooth optimization problems of the form
\begin{equation}\label{eq:generic-proxlinear}
	\min_{z\in \mathcal{Z}} \; c_1(z) + c_3(c_2(z)),
\end{equation}
where $\mathcal{Z}$ is a closed convex set, $c_1:\mathbb{R}^m\to\mathbb{R}$ and $c_2:\mathbb{R}^m\to\mathbb{R}^n$ are Lipschitz smooth, and $c_3:\mathbb{R}^n\to\mathbb{R}$ is convex but possibly nonsmooth. The idea of prox-linear method is to linearize both $c_1(z)$ and $c_2(z)$ at a given iterate to solve the resulting subproblem.

To be specific, we specialize the framework of \cite{drusvyatskiy2019efficiency} to the norm-gradient penalized problem \eqref{eq:BP-pen}. 
For this problem, we denote $z=(x,y)$ and set
\[
c_1(z)=f(x,y),\qquad
c_2(z)=\nabla_y g(x,y),\qquad
c_3(w)=\gamma\|w\|_1.
\]
This specialization is nontrivial for two reasons. Firstly, the nonsmooth term $c_3(\cdot)$ is not an arbitrary composite penalty, but the $\ell_1$-norm of the lower-level gradient, whose exact-penalty role was justified in Section~\ref{sec:penalty}. Secondly, this particular structure will later allow us, in the simple bilevel setting, to derive an explicit dual reformulation of the subproblem as a box-constrained quadratic program (See Section \ref{sec:dual-subproblem}).

Let $\mathcal Z:=C\times\mathbb R^{d_y}$. To solve \eqref{eq:EP_gammak_0}, we apply the prox-linear method with the
penalty parameter $\gamma_k$ fixed. More precisely, at the $k$-th penalty
stage, we initialize
$
z^{k,0}:=z^k. 
$
Then, for the fixed $\gamma_k$, we perform several prox-linear
iterations indexed by $j=0,1,2,\ldots$. At the $j$-th prox-linear iteration,
we linearize both $f$ and $\nabla_y g$ at
$
z^{k,j}=(x^{k,j},y^{k,j}). 
$

The prox-linear model of $F_{\gamma_k}$ is given by
\begin{equation}\label{BP-P-sub}
	\begin{aligned}
		\mathcal L_{\gamma_k,\lambda}(z;z^{k,j})
		:=
		&\, f(z^{k,j})
		+
		\nabla f(z^{k,j})^\top (z-z^{k,j}) 
		+
		\gamma_k
		\left\|
		\nabla_y g(z^{k,j})
		+
		\nabla(\nabla_y g)(z^{k,j})^\top
		(z-z^{k,j})
		\right\|_1  \\
		&+
		\frac{1}{2\lambda}\|z-z^{k,j}\|_2^2 .
	\end{aligned}
\end{equation}
Here,
\[
\nabla(\nabla_y g)(z^{k,j})
:=
\begin{bmatrix}
	\nabla_{xy}g(x^{k,j},y^{k,j})\\
	\nabla_{yy}g(x^{k,j},y^{k,j})
\end{bmatrix}.
\]
The next prox-linear iterate is obtained by solving
\begin{equation}\label{eq:exact-model-solution}
	z^{k,j+1}
	\in
	\arg\min_{z\in\mathcal Z}
	\mathcal L_{\gamma_k,\lambda}(z;z^{k,j}).
\end{equation}
For fixed $\gamma_k$, the prox-linear loop is terminated when
\begin{equation}\label{eq:fixed-gamma-stop-general}
	r_s^{k,j+1}
	:=
	\frac{1}{\lambda}
	\|z^{k,j+1}-z^{k,j}\|_2
	\leq
	\epsilon_s^{\rm in}. 
\end{equation}
Let $J_k$ be the final prox-linear index at the $k$-th penalty stage. We then
set
\begin{equation}\label{eq:outer-update-from-inner-general}
	z^{k+1}:=z^{k,J_k}.
\end{equation}

Thus, the overall idea is as follows. We borrow the general prox-linear framework from \cite{drusvyatskiy2019efficiency}, but apply it to the norm-gradient penalized problem \eqref{eq:BP-pen}. 
The special structure of this subproblem, especially in the simple bilevel case, is what distinguishes our method from a generic prox-linear scheme and leads to the dual reformulation and SPG subsolver developed later.

We emphasize that our contribution is to borrow the prox-linear algorithm to the norm-gradient penalized problem \eqref{eq:BP-pen},  and to exploit the specific $\ell_1$ norm-gradient structure in the subsequent algorithm design. 
The resulting algorithm will be referred to as the exact-penalty prox-linear method (EPPL).

The details of EPPL are given in Algorithm~\ref{alg:eep-fixed-gamma}.

\begin{algorithm}[H]
	\caption{EPPL for \eqref{eq:EP_gammak_0}}
	\label{alg:eep-fixed-gamma}
	\begin{algorithmic}[1]
		
		\State \textbf{Input:}  $\gamma_k$, 
		$z^{k,0}:=z^k$, $\lambda>0$, 
		$\epsilon_s^{\rm in}>0$. 
		
		\State Set $j:=0$.
		
		\State \textbf{while} 
		$r_s^{k,j}>\epsilon_s^{\rm in}$, \textbf{do}
		
		\State \quad Compute $	\nabla f(z^{k,j}),\ 
		\nabla_y g(z^{k,j}),\ 
		\nabla(\nabla_y g)(z^{k,j})$.
		
		\State \quad Solve the prox-linear subproblem \eqref{BP-P-sub}
		to get $z^{k,j+1}$. 
		
		
		\State \quad Set $j:=j+1$.
		
		\State \textbf{end while}
		
		\State Set $J_k:=j$ and $z^{k+1}:=z^{k,J_k}$.
		
		\State \textbf{Output:} $z^{k+1}$.
		
	\end{algorithmic}
\end{algorithm}

\subsection{Convergence Analysis of the  EPPL Method}
In this subsection, we analyze the convergence behavior of the prox-linear
iterations under a fixed penalty parameter. More precisely, we fix an
arbitrary penalty stage $\bar k$ in the outer exact-penalty continuation and
set
$
\gamma:=\gamma_{\bar k}. 
$

The update of the penalty parameter is not considered in this subsection.
Instead, we study the fixed-penalty prox-linear loop in
Algorithm~\ref{alg:eep-fixed-gamma}. To avoid double indices, we write
$
z^j:=z^{\bar k,j},
\ 
z^{j+1}:=z^{\bar k,j+1}, 
$
and denote the fixed penalized objective simply by
$
F_\gamma(z):=f(z)+\gamma p(z). 
$
Thus, all estimates below are understood for the fixed value
$\gamma=\gamma_{\bar k}$.

We first introduce the following regularity assumption commonly made in the convergence analysis of the gradient-based bilevel optimization methods.

\begin{Assumption}
	\label{ass:smooth}
	There exist constants $L_f>0$, $L_g>0$, and $L_{g,2}>0$ such that the
	following conditions hold:
	\begin{enumerate}
			\item[(i)]
			The upper-level objective $f(x,y)$ is $L_f$-Lipschitz continuously
			differentiable on $C\times \mathbb R^{d_y}$, i.e.,
			\[
			\|\nabla f(z)-\nabla f(z')\|
			\le L_f \|z-z'\|,
			\qquad \forall\, z,z'\in C\times \mathbb R^{d_y}.
			\]		
			\item[(ii)]
			The lower-level objective $g(x,y)$ is $L_g$-Lipschitz smooth with
			respect to $(x,y)$ on $C\times \mathbb R^{d_y}$, i.e.,
			\[
			\|\nabla g(z)-\nabla g(z')\|
			\le L_g \|z-z'\|,
			\qquad \forall\, z,z'\in C\times \mathbb R^{d_y}.
			\]		
			\item[(iii)]
			The gradient mapping $\nabla_y g(x,y)$ is $L_{g,2}$-Lipschitz smooth
			with respect to $(x,y)$ on $C\times \mathbb R^{d_y}$, i.e.,
			\[
			\|\nabla (\nabla_y g)(z)-\nabla (\nabla_y g)(z')\|
			\le L_{g,2} \|z-z'\|,
			\qquad \forall\, z,z'\in C\times \mathbb R^{d_y}.
			\]
		\end{enumerate}
\end{Assumption}

For the fixed penalty parameter
$\gamma>0$, the prox-linear step at the $j$-th inner iteration is
defined by
\begin{equation}\label{eq:exact-model-solution}
	\bar z^{j+1}
	=
	\arg\min_{z\in\mathcal Z}
	\mathcal L_{\gamma,\lambda}(z;z^j).
\end{equation}

The commonly used stationarity metric in nonsmooth composite optimization
\cite{drusvyatskiy2019efficiency} is the prox-gradient mapping associated with
the prox-linear model. In the present fixed-$\gamma$ setting, for
$\lambda>0$, we define
\begin{equation}\label{eq:G}
	\mathcal G_{\gamma,\lambda}(z^j)
	:=
	\lambda^{-1}
	\left(
	z^j
	-
	\arg\min_{z\in\mathcal Z}
	\mathcal L_{\gamma,\lambda}(z;z^j)
	\right)
	=
	\lambda^{-1}(z^j-\bar z^{j+1}).
\end{equation}

Our analysis follows the standard prox-linear paradigm in nonsmooth composite optimization. 
More precisely, we first quantify the approximation error between the true penalized objective and its local model, then show that each prox-linear step yields sufficient decrease, and finally combine these two ingredients to derive a sublinear bound on the prox-gradient mapping.

\begin{Remark}
	The convergence analysis below concerns the fixed-penalty prox-linear loop.
	The outer continuation strategy is used to increase the penalty parameter
	until a sufficiently large value is reached, while the stationarity estimate
	below explains the behavior of the prox-linear iterations once $\gamma$ is
	fixed.
\end{Remark}

The next lemma shows that $\mathcal L(\cdot;w)$ is a second-order accurate approximation of the $F_\gamma(\cdot)$ around the reference point $w$.

\begin{Lemma}\label{lem:model-error}
	Let $\mathcal Z:=C\times\mathbb R^{d_y}$ and suppose that
	Assumption~\ref{ass:smooth} holds. Define
	\begin{equation}\label{eq:L_gamma}
	L_\gamma:=L_f+\gamma\sqrt{d_y}L_{g,2}.
  \end{equation}
	Then, for any $z,w\in\mathcal Z$, it holds that
	\begin{equation}\label{eq:model-error-corrected}
		-\frac{L_\gamma+\lambda^{-1}}{2}\|z-w\|^2
		\le
		F_\gamma(z)-\mathcal L(z;w)
		\le
		\frac{L_\gamma-\lambda^{-1}}{2}\|z-w\|^2.
	\end{equation}
\end{Lemma}

\begin{proof}
We estimate separately the approximation errors of the smooth term $f$ and the
composite penalty term.
Since $f$ is $L_f$-smooth, we have
\[
-\frac{L_f}{2}\|z-w\|^2
\le
f(z)-\bigl(f(w)+\nabla f(w)^\top(z-w)\bigr)
\le
\frac{L_f}{2}\|z-w\|^2.
\]
Let
$
h(z):=\nabla_y g(z),
\ 
J(w):=\nabla h(w)=\nabla(\nabla_y g(w)).
$
Then
\[
\bigl|\|h(z)\|_1-\|h(w)+J(w)^\top(z-w)\|_1\bigr|
\le
\|h(z)-h(w)-J(w)^\top(z-w)\|_1.
\]
By Assumption~\ref{ass:smooth}(iii), the Taylor remainder satisfies
\[
\|h(z)-h(w)-J(w)^\top(z-w)\|_2
\le
\frac{L_{g,2}}{2}\|z-w\|^2.
\]
Therefore,
\[
\|h(z)-h(w)-J(w)^\top(z-w)\|_1
\le
\sqrt{d_y}\|h(z)-h(w)-J(w)^\top(z-w)\|_2
\le
\frac{\sqrt{d_y}L_{g,2}}{2}\|z-w\|^2.
\]
It follows that
\[
-\frac{\gamma\sqrt{d_y}L_{g,2}}{2}\|z-w\|^2
\le
\gamma\|h(z)\|_1
-
\gamma\|h(w)+J(w)^\top(z-w)\|_1
\le
\frac{\gamma\sqrt{d_y}L_{g,2}}{2}\|z-w\|^2.
\]
Combining the estimates for $f$ and the penalty term, and recalling that
$\mathcal L(z;w)$ contains the proximal term
$
\frac{1}{2\lambda}\|z-w\|^2, 
$
we obtain
\[
-\frac{L_\gamma}{2}\|z-w\|^2
\le
F_\gamma(z)-\mathcal L(z;w)+\frac{1}{2\lambda}\|z-w\|^2
\le
\frac{L_\gamma}{2}\|z-w\|^2.
\]
Rearranging gives \eqref{eq:model-error-corrected}.
\end{proof}

The next lemma is the key descent estimate for the EPPL method. It shows that, despite solving the subproblem only approximately, each iteration still decreases the penalized objective up to the controllable error $\delta_k$, and it will serve as the main ingredient in the convergence theorem below.
\begin{Lemma}\label{lem:descent-inexact}
	Suppose Assumption~\ref{ass:smooth} holds and
	$\lambda \le \frac{1}{L_\gamma}. $
	Then the iterates generated by Algorithm~\ref{alg:eep-fixed-gamma} satisfy
	\begin{equation}\label{eq:descent-inexact}
		F_\gamma(z^j)
		\ge
		F_\gamma(z^{j+1})
		-\delta_j
		+
		\frac{\lambda}{2}\|\mathcal G_{\gamma,\lambda}(z^j)\|^2.
	\end{equation}
\end{Lemma}

\begin{proof}
Since $\mathcal L(\cdot;z^j)$ is $(1/\lambda)$-strongly convex on $\mathcal Z$, and $\bar z^{j+1}$ is its minimizer over $\mathcal Z$, we have
\[
\mathcal L(z^j;z^j)
\ge
\mathcal L(\bar z^{j+1};z^j)
+
\frac{1}{2\lambda}\|z^j-\bar z^{j+1}\|^2.
\]
Recalling that $\mathcal L(z^j;z^j)=F_\gamma(z^j)$ and using \eqref{eq:G}, this becomes
\begin{equation}\label{eq:strong-convex-step}
	F_\gamma(z^j)
	\ge
	\mathcal L(\bar z^{j+1};z^j)
	+
	\frac{\lambda}{2}\|\mathcal G_{\gamma,\lambda}(z^j)\|^2.
\end{equation}

Next, by the definition of  $\delta_j$-approximate solution, 
\[
\mathcal L(\bar z^{j+1};z^j)
\ge
\mathcal L(z^{j+1};z^j)-\delta_j.
\]
Substituting this into \eqref{eq:strong-convex-step} yields
\[
F_\gamma(z^j)
\ge
\mathcal L(z^{j+1};z^j)-\delta_j
+
\frac{\lambda}{2}\|\mathcal G_{\gamma,\lambda}(z^j)\|^2.
\]

Finally, by Lemma~\ref{lem:model-error} in the rearranged form, with $z=z^{j+1}$ and $w=z^j$, and using the stepsize condition $\lambda^{-1}\ge L_\gamma$, we obtain
\[
F_\gamma(z^{j+1})\le \mathcal L(z^{j+1};z^j).
\]
Combining the last two inequalities proves \eqref{eq:descent-inexact}.
\end{proof}

Now we can derive the main convergence guarantee for Algorithm~\ref{alg:eep-fixed-gamma}. 

\begin{Theorem}\label{thm:conv-inexact}
	Suppose Assumption~\ref{ass:smooth} holds and $\lambda$ satisfies $\lambda \le \frac{1}{L_\gamma}$. Let
	$
	F_\gamma^{\inf}:=\inf_{z\in\mathcal Z}F_\gamma(z)>-\infty. 
	$
	Then the iterates $\{z^j\}$ generated by Algorithm~\ref{alg:eep-fixed-gamma} satisfy
	\begin{equation}\label{eq:rate-inexact}
		\min_{j=0,\dots,J-1}\|\mathcal G_{\gamma,\lambda}(z^j)\|^2
		\le
		\frac{2\Big(F_\gamma(z^0)-F_\gamma^{\inf}+\sum_{j=0}^{J-1}\delta_j\Big)}{\lambda J}.
	\end{equation}
	In particular, if
	$
	\sum_{j=0}^{\infty}\delta_j<\infty,
	$
	then
	$
	\min_{j=0,\dots,J-1}\|\mathcal G_{\gamma,\lambda}(z^j)\|
	=
	\mathcal O(J^{-1/2}).
	$
\end{Theorem}

\begin{proof}
Summing \eqref{eq:descent-inexact} over $j=0,\dots,J-1$ yields
\[
\sum_{j=0}^{J-1}\bigl(F_\gamma(z^j)-F_\gamma(z^{j+1})\bigr)
\ge
\sum_{j=0}^{J-1}\left(
-\delta_j+\frac{\lambda}{2}\|\mathcal G_{\gamma,\lambda}(z^j)\|^2
\right).
\]
Since the left-hand side telescopes, we obtain
\[
F_\gamma(z^0)-F_\gamma(z^J)
\ge
-\sum_{j=0}^{J-1}\delta_j
+
\frac{\lambda}{2}\sum_{j=0}^{J-1}\|\mathcal G_{\gamma,\lambda}(z^j)\|^2.
\]
Rearranging terms gives
\[
\frac{\lambda}{2}\sum_{j=0}^{J-1}\|\mathcal G_{\gamma,\lambda}(z^j)\|^2
\le
F_\gamma(z^0)-F_\gamma(z^J)+\sum_{j=0}^{J-1}\delta_j.
\]
Using the lower bound $F_\gamma(z^J)\ge F_\gamma^{\inf}$, we further have
\[
\frac{\lambda}{2}\sum_{j=0}^{J-1}\|\mathcal G_{\gamma,\lambda}(z^j)\|^2
\le
F_\gamma(z^0)-F_\gamma^{\inf}+\sum_{j=0}^{J-1}\delta_j.
\]
Dividing both sides by $\lambda J/2$ and using
\[
\min_{j=0,\dots,J-1}\|\mathcal G_{\gamma,\lambda}(z^j)\|^2
\le
\frac{1}{J}\sum_{j=0}^{J-1}\|\mathcal G_{\gamma,\lambda}(z^j)\|^2
\]
yields \eqref{eq:rate-inexact}. 

Finally, if $\sum_{j=0}^{\infty}\delta_j<\infty$, then the partial sums $\sum_{j=0}^{J-1}\delta_j$ are uniformly bounded in $J$, and the estimate \eqref{eq:rate-inexact} implies
\[
\min_{j=0,\dots,J-1}\|\mathcal G_{\gamma,\lambda}(z^j)\|
=
\mathcal O(J^{-1/2}).
\]
\end{proof}

\section{EPPL for Simple Bilevel Problem}\label{sec:simple}
We now specialize our EPPL to the simple bilevel problem. 
This setting serves as a natural starting point for two reasons. 
Firstly, by removing explicit feasible-set constraints, it isolates the essential difficulty caused by the exact-penalty term $\|\nabla g(x)\|_1$ and allows us to focus on the core mechanism of the proposed method. 
Secondly, in this setting, the prox-linear subproblem admits an explicit dual reformulation as a box-constrained quadratic program, which in turn enables an efficient dual first-order solver with closed-form primal recovery.

We consider the simple bilevel problem: 
\begin{equation}\label{eq:simplebilevel}
	\begin{aligned}
		\min_{x\in U}\quad & F(x) \\
		\text{s.t.}\quad & x \in \arg\min_{z\in\mathbb{R}^{d_x}} G(z).
	\end{aligned}
\end{equation}

Here, $U\subseteq \mathbb R^{d_x}$ is a nonempty closed convex set and 
$S:=\arg\min_{z\in\mathbb R^{d_x}}G(z)$ is assumed to be nonempty with 
$S\subseteq U$. As a direct specialization of the assumptions introduced for
the general bilevel problem, we assume: 
(i) $G$ satisfies the Polyak--\L ojasiewicz inequality on $U$;
(ii) $F$ is Lipschitz continuous on $U$;
(iii) $F$, $G$, and $\nabla G$ are Lipschitz smooth on $U$.
Under assumption (i), problem \eqref{eq:simplebilevel} is equivalent to the following gradient-based reformulation:
\begin{equation}\label{eq:SBP2}
	\min_{x\in U}\ F(x)
	\quad \text{s.t.}\quad
	\nabla G(x)=0.
\end{equation}
Under assumptions (i)-(iii), the exact-penalty analysis developed in the Theorem \ref{prop:l1-global} remains valid in the present simple bilevel setting. 
Thus, for sufficiently large $\gamma>0$, problem \eqref{eq:SBP2} is equivalent to the penalized problem
\begin{equation}\label{eq:penalized_problem}
	\min_{x\in U} \quad \Phi_\gamma(x):=F(x)+\gamma P(x), 
\end{equation}
where $	P(x):=\|\nabla G(x)\|_1. $

The details of the Exact-Penalty Method are given in Algorithm \ref{alg:outer-continuation}. 

\begin{algorithm}[H]
	\caption{Exact-Penalty Method for \eqref{eq:penalized_problem}}\label{alg:outer-continuation}
	\begin{algorithmic}[1]
		
		\State \textbf{Input:} $x^0$ and $\gamma_0>0$, $\tau>1$, $\varepsilon_{f}>0$, $\varepsilon_{s}>0$. 
		
		\State Set $k:=0$.
		
		\State \textbf{while} 	$R_{f}^{k}> \varepsilon_{f},\ \text{or}\ 
		R_{s}^{k}> \varepsilon_{s}$, \textbf{do}
		
		\State \quad Solve 
		\begin{equation}\label{eq:EP_gammak}
			\min_{x\in U} \ \Phi_{\gamma_{k}}:=F(x)+\gamma_k P(x) 
		\end{equation}
	    \quad to get $x^{k+1}$.
		
		\State \quad Set 
		$\gamma_{k+1}:=\tau \gamma_{k}$ and $k:=k+1$
		
		\State \textbf{end while}
		
		\State \textbf{Output:} $x_{\rm opt}:=x^{k}$.
		
	\end{algorithmic}
\end{algorithm}

We terminate the algorithm by combining a feasibility residual for the gradient equation
in \eqref{eq:SBP2} and a stationarity residual. 
\begin{equation}\label{eq:EPPLSBP-residuals}
	R_{f}^{k+1}
	:=
	\|\nabla G(x^{k+1})\|_1,
	\qquad
	R_{s}^{k+1}
	:=
	\frac{1}{\lambda_k}\|x^{k+1}-x^k\|_2 .
\end{equation}
The lower-level feasibility residual measures the violation of the equality constraint
$\nabla G(x)=0$ in \eqref{eq:SBP2}, while the stationarity residual  measures the
stationarity of the  step. 
The algorithm stops once
\begin{equation}\label{eq:EPPLSBP-stop}
	R_{f}^{k+1}\le \varepsilon_{f},
	\qquad
	R_{s}^{k+1}\le \varepsilon_{s}.
\end{equation}
If \eqref{eq:EPPLSBP-stop} is not satisfied, the penalty parameter is increased.

To solve \eqref{eq:EP_gammak}, we apply the prox-linear method with the
penalty parameter $\gamma_k$ fixed. More precisely, at the $k$-th penalty
stage, we initialize
$
x^{k,0}:=x^k. 
$
Then, for the fixed value $\gamma_k$, we perform several prox-linear
iterations indexed by $j=0,1,2,\ldots$. At the $j$-th prox-linear iteration,
the model of $\Phi_{\gamma_k}$ at $x^{k,j}$ is defined by
\begin{equation}\label{eq:model}
	\begin{aligned}
		\ell_{\gamma_k,\lambda}(x;x^{k,j})
		:=
		&\, F(x^{k,j})
		+\nabla F(x^{k,j})^\top(x-x^{k,j}) 
		+\gamma_k
		\Big\|
		\nabla G(x^{k,j})
		+\nabla^2G(x^{k,j})(x-x^{k,j})
		\Big\|_1  \\
		&+\frac{1}{2\lambda}\|x-x^{k,j}\|_2^2 .
	\end{aligned}
\end{equation}
The next prox-linear iterate is obtained by solving
\begin{equation}\label{eq:subproblem}
	x^{k,j+1}
	\in
	\arg\min_{x\in U}
	\ell_{\gamma_k,\lambda}(x;x^{k,j}).
\end{equation}
The fixed-$\gamma_k$ prox-linear loop is terminated when
\begin{equation}\label{eq:fixed-gamma-stop}
	R_{s}^{k,j+1}
	:=
	\frac{1}{\lambda}
	\|x^{k,j+1}-x^{k,j}\|_2
	\leq
	\epsilon_{\rm in}. 
\end{equation}
Let $J_k$ be the final prox-linear iteration index at the $k$-th penalty stage.
We then set
\begin{equation}\label{eq:outer-update-from-inner}
	x^{k+1}:=x^{k,J_k}.
\end{equation}

The details of the EPPL for simple bilevel problem (EPPL-SBP) are given in Algorithm  \ref{alg:fixed-gamma-proxlinear}

\begin{algorithm}[H]
	\caption{EPPL-SBP for \eqref{eq:EP_gammak}}
	\label{alg:fixed-gamma-proxlinear}
	\begin{algorithmic}[1]
		
		\State \textbf{Input:} $\gamma_k$,  $x^{k,0}:=x^k$, and $\lambda>0$, $\epsilon_{\rm in}>0$. 
		
		\State Set $j:=0$.
		
		\State \textbf{while} $	R_{s}^{k,j}> 
		\epsilon_{\rm in}$, \textbf{do}

			\State \quad Compute $			\nabla F(x^{k,j}),\ 
			\nabla G(x^{k,j}),\ 
			\nabla^2G(x^{k,j}).$

		\State \quad Solve the prox-linear subproblem \eqref{eq:subproblem}
	to get $	x^{k,j+1}$.

		\State \quad Set $j:=j+1$.
		
		\State \textbf{end while}
		
		\State Set $J_k:=j$ and $x^{k+1}:=x^{k,J_k}$.
		
		\State \textbf{Output:} $x^{k+1}$.  
		
	\end{algorithmic}
\end{algorithm}

\section{Solving Dual Problems Via  Spectral Projected Gradient}\label{sec:dual-subproblem}
In this section, we derive the dual reformulation of \eqref{eq:model}. 
For the dual reformulation developed below, we specialize to the unconstrained case $U=\mathbb R^{d_x}$. In this case, the prox-linear subproblem admits a
box-constrained quadratic dual reformulation.

For notational simplicity, throughout this section we fix one penalty stage
$k$ and suppress this outer index. Hence, the penalty parameter
$\gamma_k$ is denoted simply by $\gamma$, and the prox-linear iterates
$x^{k,j}$ and $x^{k,j+1}$ are denoted by $x^j$ and $x^{j+1}$, respectively.
Thus, all quantities in this section are associated with a fixed penalty
parameter $\gamma$ and a local prox-linear iteration index $j$.

Ignoring the constant term $F(x^j)$, consider
\begin{equation}\label{eq:cp_subprob_original}
	\min_{x\in\mathbb{R}^{d_x}}\ 
	\frac{1}{2\lambda}\|x-x^j\|_2^2
	+\nabla F(x^j)^\top(x-x^j)
	+\gamma\Big\|\nabla G(x^j)+\nabla^2 G(x^j)(x-x^j)\Big\|_1.
\end{equation}
Define
\[
v^j:=x^j-\lambda\nabla F(x^j),\qquad
B^j:=\nabla^2 G(x^j),\qquad
a^j:=\nabla G(x^j)-B^j x^j.
\]
Dropping constants independent of $x$ and completing the square, \eqref{eq:cp_subprob_original} is equivalent to
\begin{equation}\label{eq:primal_standard_unconstrained}
	\min_{x\in\mathbb{R}^{d_x}}\ 
	\frac{1}{2\lambda}\|x-v^j\|_2^2+\gamma\|a^j+B^j x\|_1.
\end{equation}

\begin{Remark}
In the remainder of this section, we focus on the unconstrained upper-level
case $U=\mathbb R^{d_x}$. This setting covers the simple bilevel model used in
our numerical experiments and leads to a box-constrained quadratic dual subproblem. 
For a general closed set $U$, the subproblem can be written as
\begin{equation*}\label{eq:primal_standard_with_U}
	\min_{x\in\mathbb{R}^{d_x}}\ 
	\frac{1}{2\lambda}\|x-v^j\|_2^2
	+\gamma\|a^j+B^j x\|_1
	+\delta_U(x),
\end{equation*}
where $\delta_U$ is the indicator function of $U$, namely,
\[
\delta_U(x)=
\begin{cases}
	0, & x\in U,\\
	+\infty, & x\notin U.
\end{cases}
\]
However, in this case the dual objective is generally no longer a simple quadratic function.
\end{Remark}

\subsection{Dual Subproblem and Optimality Condition}

In this subsection, we reformulate the primal  subproblem as an equivalent dual problem by using the Fenchel conjugate of the $\ell_1$ term. The following proposition summarizes the dual reformulation and its explicit quadratic form.

\begin{Proposition}\label{prop:dual-reformulation}
The dual problem of \eqref{eq:primal_standard_unconstrained} is 
\begin{equation}\label{eq:dual-problem-min}
	\min_{y\in\Omega_\gamma}\ d(y):=
	\frac{\lambda}{2}\|(B^j)^\top y\|_2^2-\langle y,c^j\rangle.
\end{equation}
where $\Omega_\gamma$ is defined by 	$\Omega_\gamma:=\{y\in\mathbb{R}^n:\|y\|_\infty\le\gamma\}$ and $
c^j:=a^j+B^j v^j.
$
Moreover, if $y^\ast$ is a dual solution, the corresponding primal solution is
recovered by
\begin{equation}\label{eq:primal-recovery-from-dual}
	x^\ast=v^j-\lambda (B^j)^\top y^\ast.
\end{equation}

\end{Proposition}

\begin{proof}
	Introduce an auxiliary variable 
	$
	r=a^j+B^jx. 
	$
	Then \eqref{eq:primal_standard_unconstrained} can be equivalently written as
	\begin{equation}\label{eq:primal-with-r}
		\begin{aligned}
			\min_{x\in\mathbb R^{d_x},\, r\in\mathbb R^{d_x}}
			\quad &
			\frac{1}{2\lambda}\|x-v^j\|_2^2+\gamma\|r\|_1  \\
			\text{s.t.}\quad &
			r=a^j+B^jx .
		\end{aligned}
	\end{equation}
	The Lagrangian associated with the equality constraint
	$r=a^j+B^jx$ is
	\[
	\mathcal{L}(x,r;y)
	=
	\frac{1}{2\lambda}\|x-v^j\|_2^2
	+\gamma\|r\|_1
	+\langle y,a^j+B^jx-r\rangle ,
	\]
	where $y\in\mathbb R^{d_x}$ is the Lagrange multiplier.
	
	The dual function is obtained by minimizing the Lagrangian over $x$ and $r$:
	\[
	q(y):=\inf_{x\in\mathbb R^{d_x},\, r\in\mathbb R^{d_x}}\mathcal{L}(x,r;y).
	\]
	First, consider the minimization with respect to $r$: 
	$
	\inf_{r\in\mathbb R^{d_x}}
	\left\{
	\gamma\|r\|_1-\langle y,r\rangle
	\right\}. 
	$
	By the conjugacy relation of the $\ell_1$-norm, we have
	\[
	\inf_{r\in\mathbb R^{d_x}}
	\left\{
	\gamma\|r\|_1-\langle y,r\rangle
	\right\}
	=
	\begin{cases}
		0, & \|y\|_\infty\le \gamma,\\
		-\infty, & \text{otherwise}.
	\end{cases}
	\]
	Therefore, the dual feasible set is 
	$
	\Omega_\gamma=\{y\in\mathbb R^{d_x}:\|y\|_\infty\le\gamma\}. 
	$
	
	Next, for $y\in\Omega_\gamma$, we minimize the remaining part with respect to
	$x$:
	\[
	\inf_{x\in\mathbb R^{d_x}}
	\left\{
	\frac{1}{2\lambda}\|x-v^j\|_2^2
	+\langle (B^j)^\top y,x\rangle
	+\langle y,a^j\rangle
	\right\}.
	\]
	The first-order optimality condition with respect to $x$ gives
	$
	\frac{1}{\lambda}(x-v^j)+(B^j)^\top y=0. 
	$
	Hence
	\[
	x=v^j-\lambda (B^j)^\top y.
	\]
	Substituting this minimizer into the Lagrangian yields
	\[
	q(y)
	=
	\langle y,a^j\rangle
	+\langle (B^j)^\top y,v^j\rangle
	-\frac{\lambda}{2}\|(B^j)^\top y\|_2^2.
	\]
	Since
$
	\langle (B^j)^\top y,v^j\rangle
	=
	\langle y,B^jv^j\rangle, 
$
	we obtain
	\[
	q(y)
	=
	\langle y,a^j+B^jv^j\rangle
	-\frac{\lambda}{2}\|(B^j)^\top y\|_2^2
	=
	\langle y,c^j\rangle
	-\frac{\lambda}{2}\|(B^j)^\top y\|_2^2.
	\]
	Thus the Lagrange dual problem is
	\[
	\max_{y\in\Omega_\gamma}
	\left\{
	\langle y,c^j\rangle
	-\frac{\lambda}{2}\|(B^j)^\top y\|_2^2
	\right\}.
	\]
	Equivalently, changing maximization into minimization, we obtain
	\[
	\min_{y\in\Omega_\gamma}
	\left\{
	\frac{\lambda}{2}\|(B^j)^\top y\|_2^2
	-\langle y,c^j\rangle
	\right\},
	\]
	which is exactly \eqref{eq:dual-problem-min}.
	
	Finally, the primal recovery formula follows from the minimizer of the
	Lagrangian with respect to $x$, namely
	\[
	x^\ast=v^j-\lambda (B^j)^\top y^\ast .
	\]
	The proof is complete.
\end{proof}

Now we study the first-order optimality condition of the dual problem derived in Proposition  \ref{prop:dual-reformulation}. 
In particular, we show that the optimality system can be equivalently written as a projected fixed-point equation, which naturally motivates projected first-order methods for solving the dual problem.

Since $d(\cdot)$ is continuously differentiable and $\Omega_\gamma$ is nonempty, closed, and convex, a point $y^\ast$ solves \eqref{eq:dual-problem-min} if and only if
\begin{equation}\label{eq:foc-normalcone}
	0\in \nabla d(y^\ast)+N_{\Omega_\gamma}(y^\ast),
\end{equation}
where $N_{\Omega_\gamma}(\cdot)$ denotes the normal cone to $\Omega_\gamma$. By the standard projection characterization, \eqref{eq:foc-normalcone} is equivalent to
\begin{equation}\label{eq:dual-fixed-point-min}
	y^\ast=\Pi_{\Omega_\gamma}\bigl(y^\ast-t\nabla d(y^\ast)\bigr),
	\qquad \forall\, t>0.
\end{equation}
This projected fixed-point form motivates the use of projected first-order methods for the dual problem, where the only nonsmoothness is handled by the projection onto the box constrained set $\Omega_\gamma$.

\subsection{Dual Spectral Projected Gradient (SPG)}

In this subsection, we solve the dual problem \eqref{eq:dual-problem-min} by a spectral projected gradient (SPG) method \cite{Birgin2000SPG}. 
The basic idea of SPG is to combine projected gradient steps with Barzilai--Borwein type spectral steplengths and a nonmonotone Armijo line search. 
This class of methods was developed for smooth optimization over simple closed convex sets, and is particularly effective when the projection onto the feasible set is easy to compute, see \cite{Birgin2001SPG}.

For a fixed outer iterate $x^j$, consider the dual minimization problem \eqref{eq:dual-problem-min}, 
its gradient is
\begin{equation}\label{eq:grad-dual-dk}
	\nabla d_j(y)=\lambda B^j(B^j)^\top y-c^j.
\end{equation}
Moreover, the projection onto $\Omega_\gamma$ is given componentwise by
\begin{equation}\label{eq:projection-box}
	\bigl(\Pi_{\Omega_\gamma}(u)\bigr)_i
	=
	\min\{\gamma,\max\{-\gamma,u_i\}\},
	\qquad i=1,\dots,n.
\end{equation}

The inner SPG solver is terminated according to the projected-gradient residual
\begin{equation}\label{eq:SPG-stop}
	r_{\mathrm{spg}}(y)
	:=
	\left\|
	\Pi_{\Omega_\gamma}\bigl(y-\eta\nabla d_j(y)\bigr)-y
	\right\|_2
	\le
	\varepsilon_{\mathrm{spg}},
\end{equation}
where $\eta>0$ is the current spectral steplength.

\begin{algorithm}[H]
	\caption{Dual SPG Inner Solver for \eqref{eq:dual-problem-min}}
	\label{alg:dual-spg}
	\begin{algorithmic}[1]
		
		\State \textbf{Input:} $B^j,\ c^j$, $\gamma>0$, $\lambda>0$,  $\varepsilon_{\rm spg}>0$. $y^0$,  $\eta_0>0$, $\eta_{\min}>0$, $\eta_{\max}>0$. 
		
		\State Set $\ell:=0$. Project the initial point onto $\Omega_\gamma$: 
		$
		y^0:=\Pi_{\Omega_\gamma}(y^0).
		$

		\State \textbf{while} $	r_{\mathrm{spg}}(y^{\ell})
		>
		\varepsilon_{\mathrm{spg}}$, \textbf{do}

		\State \quad Compute the projected search direction 
	$
			p^{\ell}
			:=
			\Pi_{\Omega_\gamma}
			\bigl(y^{\ell}-\eta_t\nabla d_j(y^{\ell})\bigr)-y^{\ell}.
	$
		
		\State \quad Find a stepsize $\alpha_{\ell}\in(0,1]$ by a nonmonotone Armijo line search.
		
		\State \quad Set $y^{{\ell}+1}:=y^{\ell}+\alpha_{\ell} p^{\ell},\ 	s^{\ell}:=y^{{\ell}+1}-y^{\ell},\ q^{\ell}:=\nabla d_j(y^{{\ell}+1})-\nabla d_j(y^{\ell}).$

		\State \quad Update 
		$
			\eta_{{\ell}+1}
			:=
			\min\{\eta_{\max},\max\{\eta_{\min},\	\frac{\langle s^{\ell},s^{\ell}\rangle}{\langle s^{\ell},q^{\ell}\rangle} \}\}.
       $
		
		\State \quad Set ${\ell}:={\ell}+1$.
		
		\State \textbf{end while}
		
		\State Set $y^\star:=y^{\ell}$. Recover the primal point by
		$
		x^{j+1}:=v^j-\lambda(B^j)^\top y^\star .
		$
		
		\State \textbf{Output:} the final dual iterate $y^\star$ and the recovered primal point $x^{j+1}$.
		
	\end{algorithmic}
\end{algorithm}

The use of SPG in Algorithm~\ref{alg:dual-spg} is motivated by three features of the dual subproblem \eqref{eq:dual-problem-min}. 
Firstly, the objective $d_j(\cdot)$ is a smooth convex quadratic function. 
Secondly, the feasible set $\Omega_\gamma$ is a simple box, so the projection \eqref{eq:projection-box} is explicit. 
Thirdly, the dual iterate can be converted back to a primal iterate by the closed-form recovery formula \eqref{eq:primal-recovery-from-dual}. 
These properties make SPG a natural and efficient inner solver within our EPPL-SBP.

We next record the convergence guarantee of the dual SPG iterations. The result is a direct consequence of the general theory of nonmonotone spectral projected gradient methods \cite{Birgin2000SPG}, specialized here to the smooth convex quadratic dual objective and the box-constrained set $\Omega_\gamma$.

\begin{Theorem}\label{thm:spg-dual-convergence}
	Let $\{y^{\ell}\}_{{\ell}\ge 0}\subset \Omega_\gamma$ be the sequence generated by Algorithm~\ref{alg:dual-spg}. Then the following statements hold.
	\begin{itemize}
		\item[(i)] The iterates are well defined and remain feasible, i.e., $y^{\ell}\in\Omega_\gamma$ for all ${\ell}\ge 0$.
		\item[(ii)] Every accumulation point $\bar y$ of $\{y^{\ell}\}$ satisfies the first-order optimality condition
		$
		0\in \nabla d_j(\bar y)+N_{\Omega_\gamma}(\bar y). 
		$
		Equivalently,
		$
		\bar y=\Pi_{\Omega_\gamma}\bigl(\bar y-t\nabla d_j(\bar y)\bigr),
		\  \forall\, t>0.
		$
		\item[(iii)] Since $d_j(\cdot)$ is a convex quadratic function and $\Omega_\gamma$ is convex, every such accumulation point is a global minimizer of \eqref{eq:dual-problem-min}.
	\end{itemize}
\end{Theorem}

\begin{proof}
Parts (i) and (ii) follow directly from \cite[Theorem~2.4]{Birgin2000SPG}. Part (iii) follows from the convexity of the quadratic objective $d_j(\cdot)$.
\end{proof}

\section{Numerical Experiments}\label{sec:numerics}

In this section, we apply our method EPPL-SBP to
Minimum Norm Solution Problem(MNP) and compare its performance with other existing methods in the literature \cite{beck2014minimal,sabach2017first,kaushik2021vi,gong2021bi,wang2024bisection}.
The numerical tests are conducted in \textsc{Matlab} R2023a on a MacBook Air(13-inch, M3, 2024) running macOS Sonoma~14.6 with an Apple M3 chip and 16~GB of memory.

\subsection{EPPL-SBP on the Minimum-Norm Problem (MNP)}

We consider the simple bilevel problem \cite{wang2024bisection}
\begin{equation}\label{eq:linreg_bilevel}
	\begin{aligned}
		\min_{x\in\mathbb{R}^n}\quad & F(x):=\frac12\|x\|_2^2 \\
		\text{s.t.}\quad & x\in \arg\min_{z\in\mathbb{R}^n}\ G(z):=\frac12\|Az-b\|_2^2.
	\end{aligned}
\end{equation}
This model seeks the minimum-norm least-squares solution.

We use the YearPredictionMSD dataset\footnote{https://archive.ics.uci.edu/dataset/203/yearpredictionmsd}, which consists of 515,345 songs released between 1992 and 2011. Each sample includes the release year together with 90 additional attributes. In our experiment, we randomly draw 1000 songs independently and uniformly from the dataset, and denote by $A$ the resulting feature matrix and by $b$ the corresponding vector of release years. 
The features are normalized to the range $[-1,1]$, an intercept term is appended, and the response vector is normalized to the range $[0,1]$. 
We then denote by $m$ and $n$ the dimensions of the resulting matrix $A\in\mathbb{R}^{m\times n}$.

The penalty parameter is initialized by $\gamma_0=100$ and is increased by
$\gamma_{k+1}
=
\tau_\gamma\gamma_k$
, where
$
\tau_\gamma=1.2. 
$
For each fixed penalty parameter $\gamma_k$, we perform at most 40 prox-linear iterations. The proximal stepsize is fixed as
$
\lambda=10^{-2}.
$
The whole continuation procedure is stopped once
\begin{equation}\label{eq:num-eppl-res}
	R_f\le \varepsilon_f=10^{-5},
	\qquad
	R_s\le \varepsilon_s=10^{-5} .
\end{equation}
is satisfied. 
If \eqref{eq:num-eppl-res} is not satisfied, we terminate the current fixed-$\gamma_k$
stage and move to the next penalty parameter.

For each prox-linear subproblem, the dual problem
\eqref{eq:dual-problem-min} is solved by a nonmonotone spectral projected
gradient method. 
Let $q$ denote the cumulative index of the prox-linear
subproblem, namely
$
q=q(k,j). 
$
To avoid oversolving early subproblems while retaining high final accuracy,
we use the adaptive inner accuracy schedule
\begin{equation}\label{eq:inner_adp_new}
	\varepsilon_{\rm spg}^{k,j}
	=
	\begin{cases}
		10^{-3}, & q(k,j)\le 15,\\[0.3em]
		10^{-4}, & 15<q(k,j)\le 50,\\[0.3em]
		10^{-6}, & q(k,j)>50,
	\end{cases}
	\qquad
	T_{\max}^{k,j}
	=
	\begin{cases}
		200, & q(k,j)\le 15,\\[0.3em]
		400, & 15<q(k,j)\le 50,\\[0.3em]
		1000, & q(k,j)>50.
	\end{cases}
\end{equation}
The SPG method stops when
$
r_{\rm spg}\le \varepsilon_{\rm spg} 
$
or when the maximum number $T_{\max}^{k,j}$ of inner iterations is reached.

\begin{figure}[htbp]
	\centering
	\includegraphics[width=0.98\textwidth]{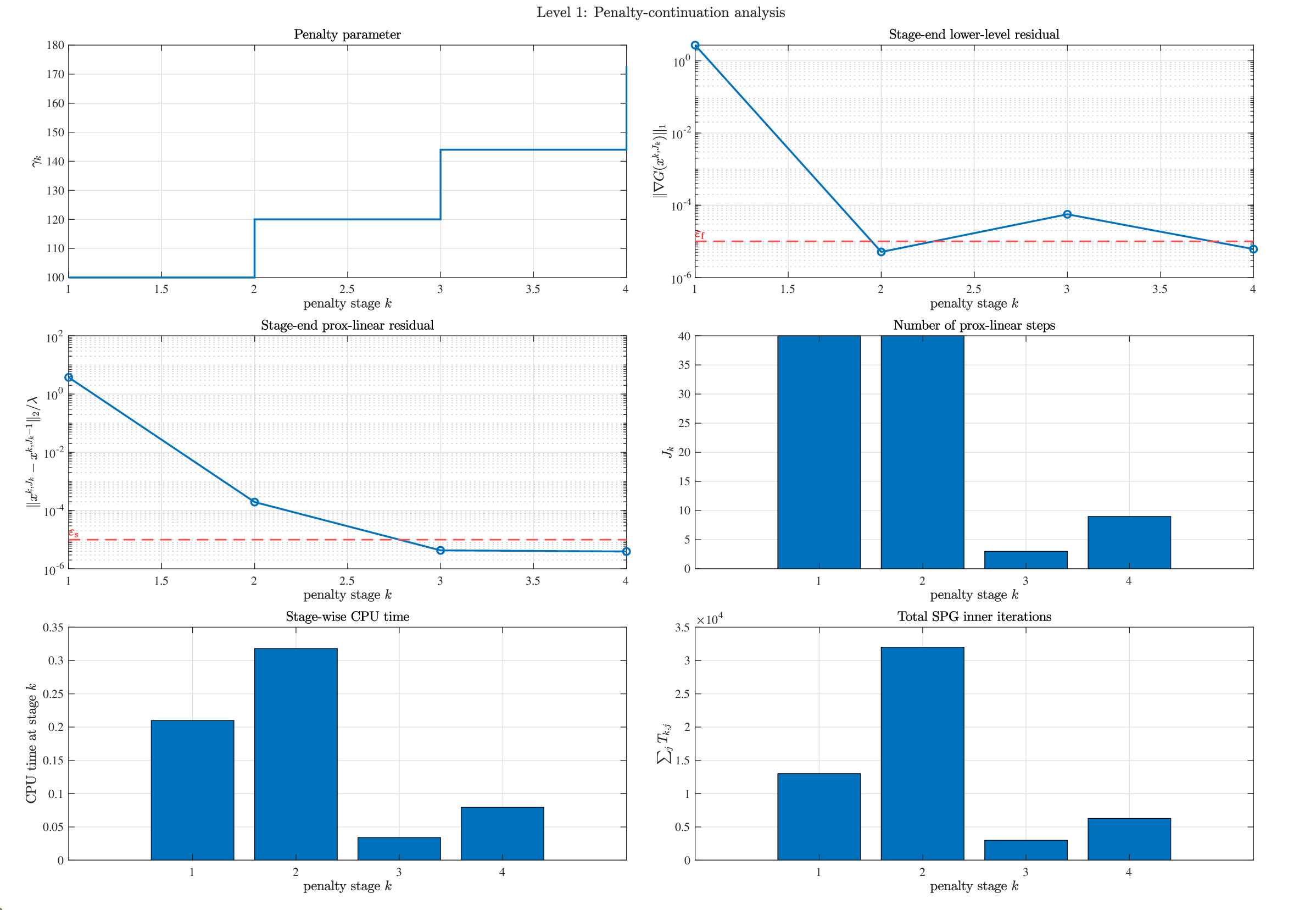}
	\caption{
		Penalty-continuation analysis. 
		The horizontal axis is the
		penalty stage $k$. The panels show the penalty parameter $\gamma_k$, the
		stage-end lower-level residual $\|\nabla G(x^{k,J_k})\|_1$, the
		stage-end prox-linear residual
		$\|x^{k,J_k}-x^{k,J_k-1}\|_2/\lambda$, the number $J_k$ of prox-linear
		steps, the CPU time at each stage, and the total inner SPG effort
		$\sum_j T_{k,j}$.
	}
	\label{fig:level1}
\end{figure}

Figure~\ref{fig:level1} illustrates the outer penalty-continuation behavior of
EPPL-SBP. The method uses only four penalty stages in this run, with 
$
\gamma_1=100,\ 
\gamma_2=120,\ 
\gamma_3=144,\ 
\gamma_4=172.8 .
$
and terminates at the fourth stage. This shows that the continuation strategy
does not need to drive the penalty parameter to a very large value, a moderate
penalty is already sufficient to enforce the lower-level stationarity condition. 
Most of the computational effort is spent in the first two stages, where the
algorithm uses the full budget of $40$ prox-linear steps. These stages are
responsible for the main reduction of the residuals and move the iterates close
to the lower-level solution set. Once this has been achieved, the remaining
stages serve mainly as refinement steps. The third and fourth stages require
only $3$ and $9$ prox-linear iterations, respectively. This demonstrates the
benefit of the continuation strategy:  early stages provide robustness, while
later stages improve feasibility and stationarity at a relatively low
additional cost.

\begin{figure}[htbp]
	\centering
	\includegraphics[width=0.98\textwidth]{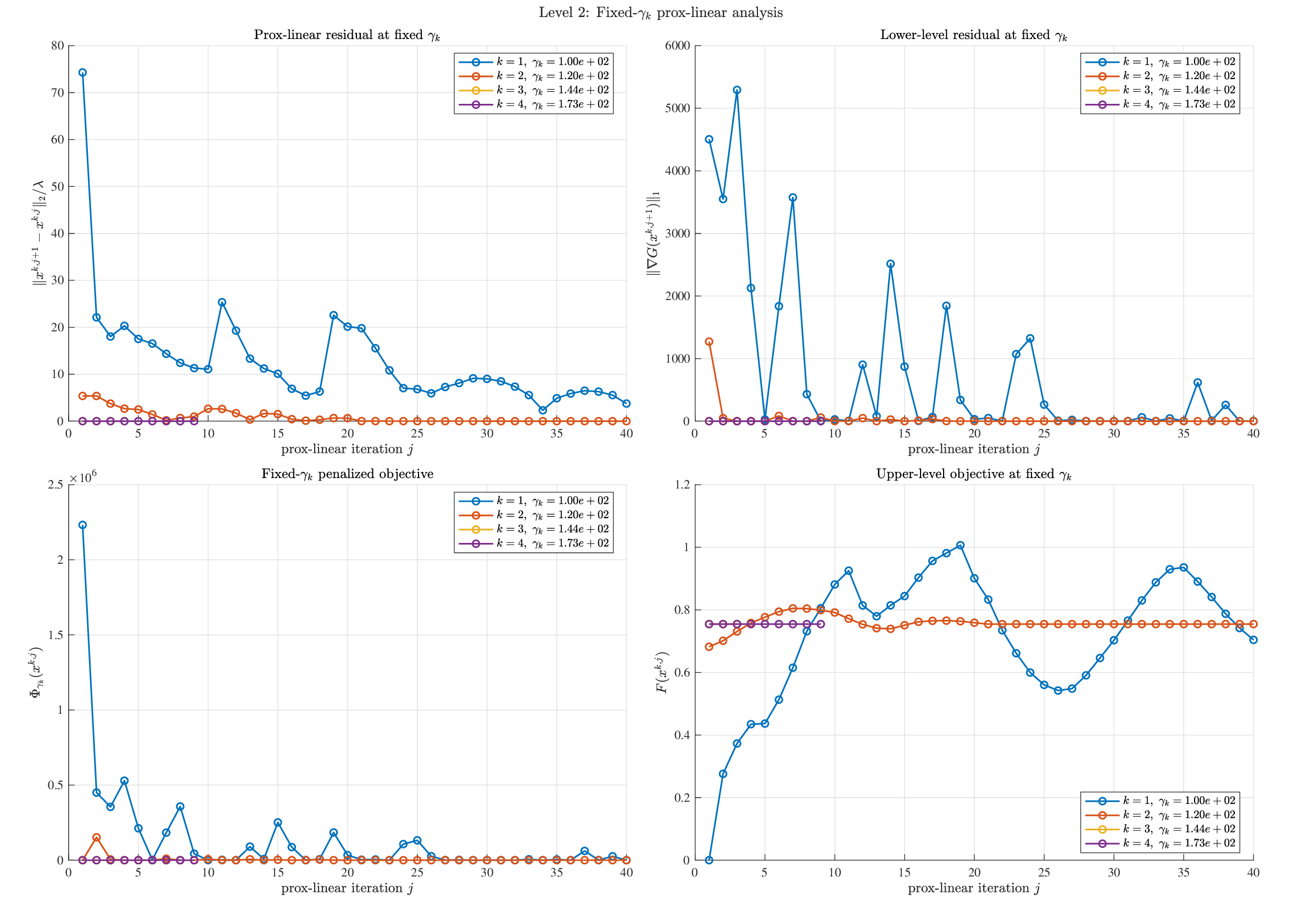}
	\caption{
		Fixed-$\gamma_k$ prox-linear analysis. The horizontal axis is
		the prox-linear iteration index $j$ within each fixed penalty stage.
		The panels show the prox-linear residual, the lower-level residual, the
		fixed-$\gamma_k$ penalized objective, and the upper-level objective.
	}
	\label{fig:level2}
\end{figure}

Figure~\ref{fig:level2} gives a stage-wise view of the fixed-$\gamma_k$
prox-linear iterations. The first two panels show that, in the early penalty
stages, both the prox-linear residual and the lower-level residual can be
large and nonmonotone. This behavior is natural because the method uses local
prox-linear models and solves the dual subproblems inexactly. These early
stages mainly serve to move the iterate toward the lower-level solution set. 
After the continuation has sufficiently increased the penalty parameter, the
prox-linear dynamics become much more stable. In the later stages, the
residuals stay at a much smaller scale, showing that only local refinement is
needed. The plots of $\Phi_{\gamma_k}(\cdot)$ and $F(\cdot)$ further show that the method is
not simply minimizing the penalized objective monotonically; rather, it
adjusts the iterate by balancing upper-level objective reduction and
lower-level stationarity enforcement under each fixed $\gamma_k$.

\begin{figure}[htbp]
	\centering
	\includegraphics[width=0.98\textwidth]{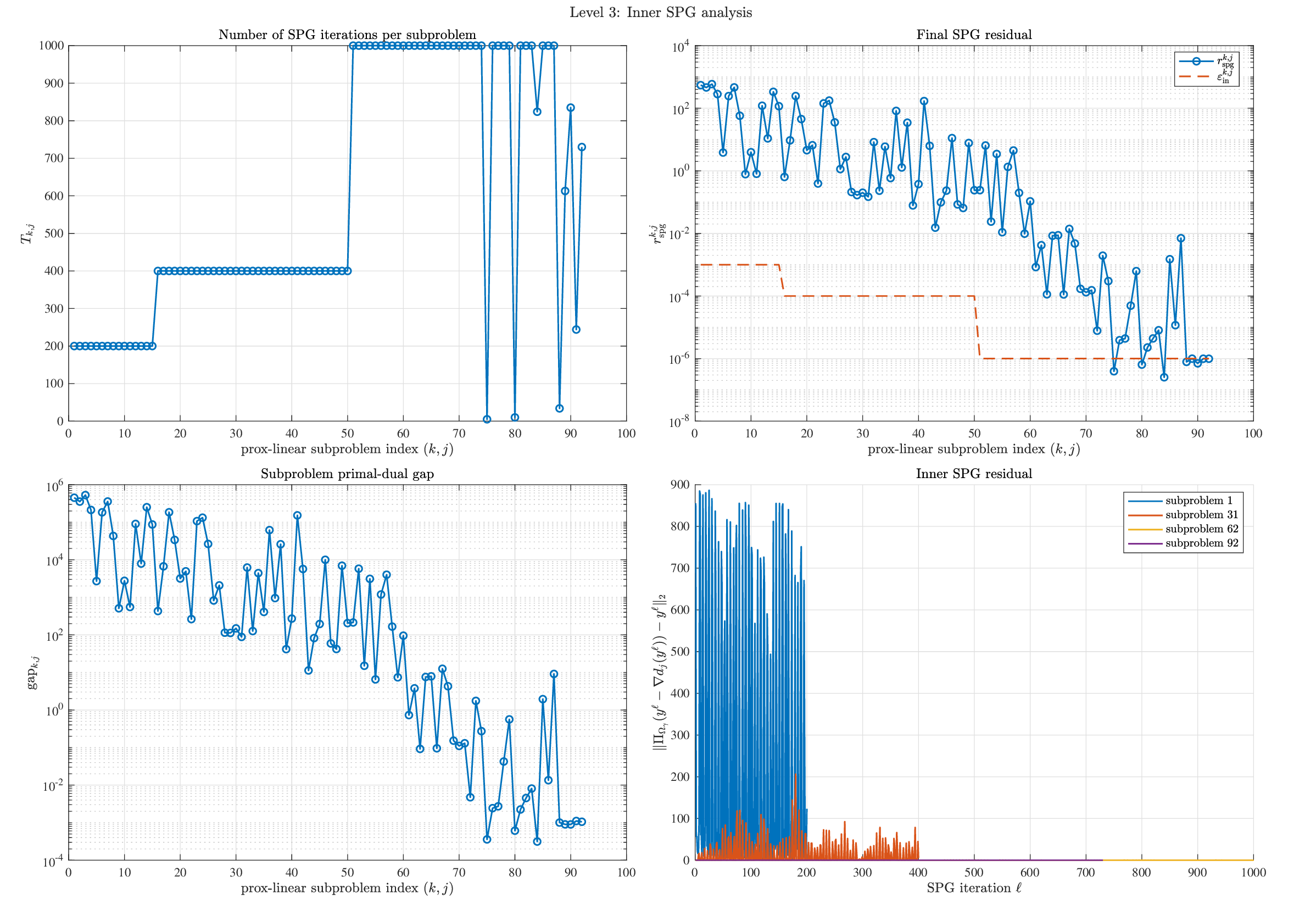}
	\caption{
		Inner SPG analysis. The horizontal axis in the first three
		panels is the prox-linear subproblem index $(k,j)$, while the horizontal
		axis in the last panel is the SPG inner iteration index $\ell$. The
		panels show the number of SPG iterations, the final SPG residual, the
		primal-dual gap of the model subproblem, and the inner SPG residual
		decay for selected subproblems.
	}
	\label{fig:level3}
\end{figure}

Figure~\ref{fig:level3} shows that the inner SPG solver follows an
adaptive accuracy strategy. Early dual subproblems often reach the prescribed
iteration budget, reflecting the fact that they are solved only inexactly and
are mainly used to produce descent directions for the prox-linear layer. As
the iterates approach the solution region, the number of SPG iterations drops
for several subproblems, showing the benefit of warm-starting and continuation. 
The final SPG residual tracks the prescribed tolerance schedule, which is
tightened from $10^{-3}$ to $10^{-4}$ and finally to $10^{-6}$. This prevents
oversolving in the early stage while ensuring high accuracy near convergence.
The primal-dual gap and the representative inner residual decay curves confirm
that the SPG solver solves the dual subproblems to the required accuracy.

The final numerical results are summarized in Table~\ref{tab:eppl-final}.
The method terminates at the fourth penalty stage with $\gamma=172.8$ after
92 prox-linear steps. At the final iterate, we obtain
\[
R_f=\|\nabla G(x^{k+1})\|_1=6.1105\times 10^{-6},
\qquad
R_s=\frac{\|x^{k+1}-x^{k}\|_2}{\lambda}
=
3.944\times 10^{-6},
\]
which satisfies the stopping criterion \eqref{eq:EPPLSBP-stop}. 
These values show that EPPL-SBP recovers a highly accurate minimum-norm
least-squares solution.

\begin{table}[htbp]
	\centering
	\caption{Final performance of EPPL-SBP on the YearPredictionMSD example.}
	\label{tab:eppl-final}
	\begin{tabular}{lc}
		\toprule
		Quantity & Value \\
		\midrule
		Final penalty parameter $\gamma$ & $1.728\times 10^{2}$ \\
		Number of penalty stages & $4$ \\
		Number of prox-linear steps & $92$ \\
		Final lower-level residual $\|\nabla G(x)\|_1$ & $6.1105\times 10^{-6}$ \\
		Final lower-level gradient norm $\|\nabla G(x)\|_2$ & $9.88\times 10^{-7}$ \\
		Final prox-linear residual & $3.944\times 10^{-6}$ \\
		Final SPG residual & $9.884\times 10^{-7}$ \\
		CPU time & $0.6109$ s \\
		\bottomrule
	\end{tabular}
\end{table}

\subsection{Comparison with Existing Methods}
We compare the performance of EPPL-SBP with several existing methods \footnote{The implementations of the benchmark methods are taken from the public repository \url{https://github.com/XuShi22/BisecBiO}. } 
	for solving \eqref{eq:linreg_bilevel}, including the minimal norm gradient method (MNG), the bilevel gradient SAM method (BiG-SAM) \cite{sabach2017first}, the averaging iteratively regularized gradient method (a-IRG) \cite{kaushik2021vi}, the dynamic barrier gradient descent method (DBGD) \cite{gong2021bi}, and the bisection-based method (Bisec-BiO) \cite{wang2024bisection}. Motivated by 
\cite{wang2024bisection}, we use the MATLAB function \texttt{lsqminnorm} to solve problem \eqref{eq:linreg_bilevel} and obtain the optimal values $g^\ast$ and $p^\ast$ for benchmarking purposes. 

The solution accuracy is measured by the prescribed tolerances
$
\epsilon_f=10^{-6},
\ 
\epsilon_g=10^{-7}. 
$
Accordingly, the lower-level and upper-level accuracy requirements are defined by
\[
G(x^k)-g^\ast\le \epsilon_g,
\qquad
|F(x^k)-p^\ast|\le \epsilon_f.
\]
In particular, a method is regarded as having reached the prescribed bilevel accuracy only if both inequalities are satisfied simultaneously.

\begin{figure}[H]
	\centering
	\begin{minipage}{0.48\linewidth}
		\centering
		\includegraphics[width=\linewidth]{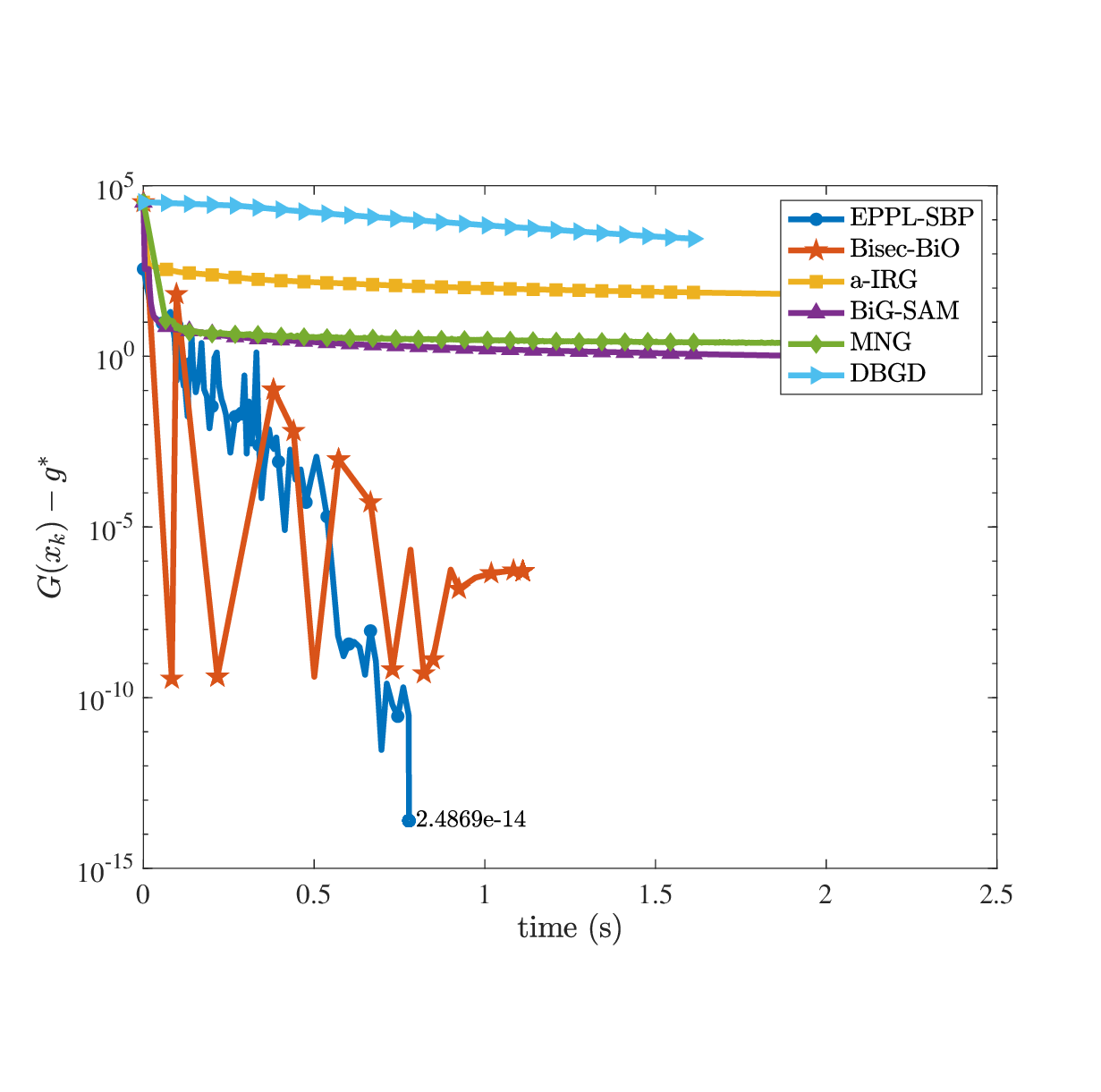}
	\end{minipage}\hfill
	\begin{minipage}{0.48\linewidth}
		\centering
		\includegraphics[width=\linewidth]{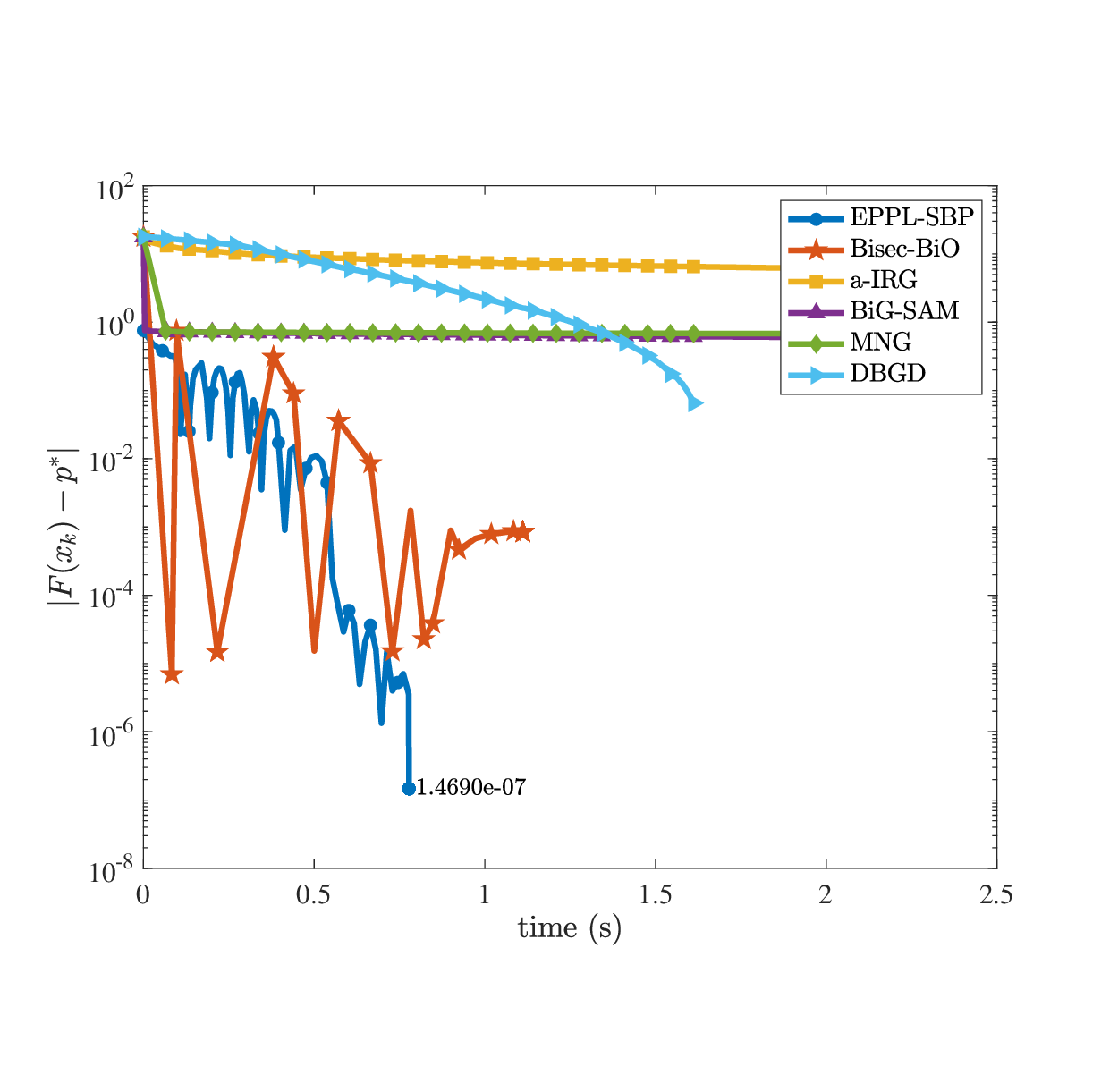}
	\end{minipage}
	\caption{
		Comparison of EPPL-SBP with Bisec-BiO, a-IRG, BiG-SAM, MNG, and DBGD. Panel (a) reports the lower-level gap $G(x^k)-g^\ast$, while panel (b) reports the upper-level gap $|F(x^k)-p^\ast|$, both plotted against running time.
		}
	\label{fig:yearprediction_compare}
\end{figure}


\begin{table}[H]
	\centering
	\caption{Final status at termination on the YearPredictionMSD linear regression instance.}
	\label{tab:final_hit_year}
		\begin{tabular}{lcccc}
	\hline
	Algorithm & Iter. & Time(s) & Lower gap & Upper gap \\
	\hline
	Bisec-BiO & 25 & 1.1114 & $5.000{\times}10^{-7}$ & $8.398{\times}10^{-4}$ \\
	a-IRG & 31333 & 2.1114 & $6.249{\times}10^{1}$ & $5.991{\times}10^{0}$ \\
	BiG-SAM & 30024 & 2.1120 & $9.620{\times}10^{-1}$ & $5.943{\times}10^{-1}$ \\
	MNG & 1411 & 2.1126 & $2.440{\times}10^{0}$ & $6.784{\times}10^{-1}$ \\
	DBGD & 15575 & 1.6114 & $2.778{\times}10^{3}$ & $6.545{\times}10^{-2}$ \\
{EPPL-SBP} & $\mathbf{76}$ & $\mathbf{0.7776}$
	& $\mathbf{2.487{\times}10^{-14}}$ & $\mathbf{1.469{\times}10^{-7}}$ \\
	\hline
\end{tabular}
\end{table}

Figure~\ref{fig:yearprediction_compare} and Table~\ref{tab:final_hit_year} summarize the comparison results. 
Among all tested methods, EPPL-SBP achieves the best overall final solution accuracy. 
Its final lower-level and upper-level gaps are
$
G(x^k)-g^\ast = 2.487\times 10^{-14},
\ 
|F(x^k)-p^\ast| = 1.469\times 10^{-7},
$
which are substantially smaller than those of all baseline methods.

In particular, Bisec-BiO attains a very small final lower-level gap,
$
G(x^k)-g^\ast = 5.000\times 10^{-7}, 
$
but its final upper-level gap remains at the order of $10^{-4}$. 
By contrast, EPPL-SBP simultaneously drives both the lower-level and upper-level gaps to much smaller values. 
Moreover, its total runtime is also competitive: EPPL-SBP terminates in $0.7776$ seconds, which is faster than all compared methods except DBGD and Bisec-BiO, while delivering significantly higher final accuracy.

The remaining methods perform substantially worse on this problem. 
At termination, a-IRG still has very large lower-level and upper-level gaps. 
BiG-SAM and MNG improve upon a-IRG, but neither method comes close to the final accuracy achieved by EPPL-SBP. 
DBGD exhibits the largest lower-level gap and also a relatively large upper-level gap. 
These results indicate that, on this minimum-norm least-squares bilevel model, the proposed EPPL-SBP framework provides the most accurate final solution among all tested methods.

\section{Conclusions}\label{sec:conclusion}

In this paper, we proposed an exact-penalty prox-linear method for bilevel optimization based on the $\ell_1$-norm of the lower-level gradient. The proposed penalty allows one to replace the implicit lower-level optimality requirement by a nonsmooth but tractable single-level objective. 
To solve the resulting penalized problem, we developed a prox-linear method and analyzed its stationarity-oriented convergence behavior. We then specialized the method to the  simple bilevel problem, where the subproblem admits an explicit dual reformulation as a box-constrained quadratic optimization problem. This structure enabled us to design a dual SPG solver with closed-form primal recovery and to establish convergence of the inner dual iterations.
Numerical experiments on the minimum-norm least-squares bilevel model 
demonstrated that the proposed method is effective in driving both the lower-level and upper-level gaps to high accuracy.

Several directions remain for future work. One is to extend the present framework to constrained simple bilevel problems and, more generally, to bilevel models with explicit lower-level constraints. 
It would also be of interest to investigate adaptive penalty update rules and more scalable inner solvers for large-scale machine learning applications.

\section*{Conflict of interest}
The authors declare that they have no conflict of interest.


\bibliographystyle{unsrt}
\bibliography{ref}
\end{document}